\documentclass[10pt]{amsart}
\usepackage{hyperref}

\makeatletter \oddsidemargin.9375in \evensidemargin \oddsidemargin
\newtheorem{theorem}{Theorem}[section]
\newtheorem{lemma}[theorem]{Lemma}

\newtheorem{corollary}[theorem]{Corollary}

\theoremstyle{definition}
\newtheorem{definition}[theorem]{Definition}

\theoremstyle{remark}

\numberwithin{equation}{section}

\begin{document}

\title[Existence and multiplicity of positive solutions]
{Existence and multiplicity of positive solutions for quasilinear  elliptic
Systems involving the p-laplacian}

\author[S. S. Kazemipoor]
{Seyyed Sadegh Kazemipoor$^1$}
\address{$^{1}$   Department of Mathematics sciences, Guilan University, Iran}

\email{S.Kazemipoor@umz.ac.ir}

\author[M. Zakeri]{Mahboobeh Zakeri$^2$}

\address{$^{2}$ Department of Mathematics sciences, Guilan University, Iran}

\email{Mzakeri.math@gmail.com}


\keywords{Critical points;
nonlinear boundary value problems; \hfill\break\indent quasilinear
p-Laplacian problem; fibering map; Nehari manifold.\\}

\begin{abstract}
 We study the existence and multiplicity of solutions to
 the elliptic system
 \begin{gather*}
 -\operatorname{div}(|\nabla u|^{p-2} \nabla u)+m_1(x)|u|^{p-2}u
 =\lambda g_u(x,u,v) \quad x\in \Omega,\\
 -\operatorname{div}(|\nabla v|^{p-2} \nabla v)+m_2(x)|v|^{p-2}v=\mu
 g_v(x,u,v)  \quad x\in \Omega,\\
|\nabla u|^{p-2}\frac{\partial u}{\partial n}=f_u(x,u,v),\quad
 |\nabla v|^{p-2}\frac{\partial v}{\partial n}=f_v(x,u,v)  \quad x\in\partial\Omega,
 \end{gather*}
 where $\Omega\subset \mathbb{{R}}^N$ is
 a bounded and smooth domain.
 Using  fibering maps and  extracting Palais-Smale sequences in the
 Nehari manifold, we prove the existence of at least two
 distinct nontrivial nonnegative solutions.
\end{abstract}

\maketitle

\section{Introduction}

In this article, we study the existence and multiplicity of positive
solutions of the  quasilinear elliptic system
\begin{equation} \label{e1.1}
\begin{gathered}
-\operatorname{div}(|\nabla u|^{p-2} \nabla u)+m_1(x)|u|^{p-2}u
 =\lambda g_u(x,u,v)   \quad x\in \Omega,\\
-\operatorname{div}(|\nabla v|^{p-2} \nabla v)+m_2(x)|v|^{p-2}v
 =\mu g_v(x,u,v)  \quad x\in \Omega,\\
|\nabla u|^{p-2} \frac{\partial u}{\partial n}=f_u(x,u,v),\quad
|\nabla v|^{p-2}\frac{\partial v}{\partial n}=f_v(x,u,v)  \quad x\in\partial\Omega,
\end{gathered}
 \end{equation}
where $\lambda, \mu>0$, $p>2$, $\Omega\subset\mathbb{{R}}^N$ is a
bounded domain in  $\mathbb{{R}}^N$ with the smooth boundary
$\partial\Omega$, $\frac{\partial}{\partial n}$ is the outer normal
derivative, $m_1, m_2 \in C(\bar{\Omega})$ are positive bounded
functions together with  the following  assumptions on the functions
$f$ and $g$ :
\begin{itemize}

\item[(A1)] $\frac{\partial^2}{\partial t^2} f(x,t|u|,t|v|)|_{t=1}\in
C(\partial\Omega \times \mathbb{R}^2)$ and for $u, v\in
L^p(\partial\Omega)$, the integral
$\int_{\partial\Omega}\frac{\partial^2}{\partial t^2}
\big(f(x,t|u|,t|v|)\big)dx$ has the same sign for every $t>0$.

\item[(A2)] There exists $C_1>0$ such that
$$
f(x,u,v)\leq\frac{1}{r} \frac{\partial}{\partial t}
f(x,tu,tv)|_{t=1}\leq\frac{1}{r(r-1)}\frac{\partial^2}{\partial t^2}
f(x,tu,tv)|_{t=1}\leq C_1(u^r+v^r),
$$
where $p<r<p^*$ ($p^*=\frac{pN}{N-p}$ if $N>2$, $p^*=\infty$ if
$N\leq p$) for all $(x,u,v)\in {\partial\Omega} \times
\mathbb{R^+}\times \mathbb{R^+}$.

\item[(A3)]  $\frac{\partial}{\partial t} f(x,tu,tv)|_{t=0}\geq0$ and
$\lim_{t\to\infty}\frac{\frac{\partial}{\partial t}
f(x,tu,tv)}{t^{p-1}}= \eta(x,u,v)$ uniformly respect to $(x,u,v)$,
where $\eta(x,u,v)\in C(\bar{\Omega}\times\mathbb{R}^2)$ and
$|\eta(x,u,v)|>\theta>0$, a.e.
 for all $(x,u,v)\in {\Omega} \times \mathbb{R^+}\times \mathbb{R^+}$.

\item[(A4)] $g_u(x,u,v), g_v(x,u,v) \in C^1({\Omega} \times \mathbb{R^+}\times \mathbb{R^+})$ such that
$g_u(x,0,0)\geq0$, $ g_v(x,u,v)\geq0$, $g_u(x,0,0)\not\equiv0$ and there exist
$C_2>0$, $C_3>0$ such that, $|g_u(x,u,v)| \leq C_2(1+u^{p-1})$ and
$|g_v(x,u,v)|\leq C_3(1+v^{p-1})$, where $x\in \Omega$, $u,v
\in{\mathbb{R^+}}$ and $p>2$.

\item[(A5)] For $u,v\in W^{1,p}(\Omega)$,
$\int_{\Omega}\frac{\partial}{\partial u} g_u(x,t|u|,t|v|)u^2dx$ and
$\int_{\Omega}\frac{\partial}{\partial v} g_v(x,t|u|,t|v|)v^2dx$ have the
same sign for every  $t>0$ and there exist $C_4>0$, $C_5>0$ such
that $|g_u(x,u,v)|\leq C_4u^{p-2}$ and $|g_v(x,u,v)|\leq C_5v^{p-2}$ for
all $(x,u,v)\in \Omega\times\mathbb{R^+}\times\mathbb{R^+}$.

\item[(A6)] There exist $ a_{i}\geq 0, c_{i} \geq 0$$(i=1,2)$ such
that
\begin{gather*}
0\leq g_u(x,u,v) \leq a_{1}|(u,v)|^{r-1} + c_{1},  0\leq g_v(x,u,v) \leq a_{2}|(u,v)|^{r-1} + c_{2}
\end{gather*}

\item[(A7)] There existence an $\epsilon^{'}>0$, $c_{3}>0$,
$p<\varsigma<p^{*}$ such that
\[
g_u(x,u,v)u + g_v(x,u,v)v\leq
(\lambda_{1}-\epsilon')(|u|^{p}+|v|^{p})+c_{3}(|u|^{\varsigma}+|v|^{\varsigma})
\]
where $\lambda_{1}$ stands for the first eigenvalue of the operator
$-\Delta_{p}$ in $W_{0}^{1,p}(\Omega)$.

\item[(A8)]  $g_u(x,u,v)$ and $g_v(x,u,v)$ also satisfies
\[
\liminf_{|(u,v)|\to \infty}\frac{g_u(x,u,v)}{|(u,v)|^{p-1}} =
+\infty,\quad \liminf_{|(u,v)|\to
\infty}\frac{g_v(x,u,v)}{|(u,v)|^{p-1}} = +\infty.
\]

\end{itemize}

Define the Sobolev space
\begin{equation} \label{e1.2}
W:=W^{1,p}(\Omega)\times W^{1,p}(\Omega),
\end{equation}
endowed with the norm
$$
\|(u,v)\|_{W}=\Big(\int_{\Omega}(|\nabla u|^p+m_1(x)|u|^p)dx
+\int_{\Omega}(|\nabla v|^p+m_2(x)|v|^p)dx\Big)^{1/p},
$$
which is equivalent to the standard norm. We use the standard
$\mathrm{L}^r(\Omega)$ spaces whose norms are denoted by
$\|u\|_{r}$. Throughout this paper, we denote $S_q$ and
$\bar{S}_{q}$
 the best Sobolev and the best Sobolev trace constants for
the embedding
 of $W^{1,p}(\Omega)$ into $\mathrm{L}^{q}(\Omega)$ and $W^{1,p}(\Omega)$ into
 $\mathrm{L}^{q}(\partial\Omega)$, respectively. So we have
\begin{equation}\label{e1.3}
 \frac{(\|(u,v)\|^p_{W})^{q}}{(\int_{\partial\Omega}(|u|^{q}+|v|^{q})dx )^p}
 \geq \frac{1}{2^p\bar{S}^{pq}_q} \quad  \text{and} \quad
\frac{(\|(u,v)\|^p_{W})^{q}}{(\int_{\Omega}(|u|^{q}+|v|^{q})dx )^p}
 \geq \frac{1}{2^p{S}^{pq}_{q}}.
\end{equation}

The purpose of this paper is to prove the following results.

\begin{theorem} \label{thm1.1}
There exists $K^*\subset (\mathbb{R}^+)^2$ such that for
 each $(\lambda,\mu)\in K^*$  problem \eqref{e1.1} has at least one
positive solution.
\end{theorem}

\begin{theorem} \label{thm1.2}
There exists $K^{**}\subset K^*$ such that for
 each $(\lambda,\mu)\in K^{**}$  problem \eqref{e1.1} has at least two
distinct positive solutions.
\end{theorem}

\begin{definition} \label{def2.1} \rm
 A pair of functions $(u,v)\in W$ ($W$ is given by \eqref{e1.2})
 is said to be a weak solution of \eqref{e1.1}, whenever
\begin{align*}
&\int_{\Omega}{\big(|\nabla u|^{p-2}\nabla u.\nabla
{\varphi}_1}+m_1(x)|u|^{p-2}u{\varphi}_1\big)dx
- \lambda\int_{\Omega}g_u(x,u,v)\varphi_1dx-\int_{\partial\Omega}{f_u(x,u,v){{\varphi}_1}}dx=0,
\\
&\int_{\Omega}{\big(|\nabla v|^{p-2} \nabla v.\nabla
{\varphi}_2}+m_2(x)|v|^{p-2}v{\varphi}_2\big)dx-\mu\int_{\Omega}g_v(x,u,v)\varphi_2dx-
\int_{\partial\Omega}{f_v(x,u,v){{\varphi}_2}}dx =0,
\end{align*}
for all $(\varphi_1,\varphi_2)\in W$.
 \end{definition}

Associated with problem \eqref{e1.1}, we consider the energy
functional $J_{\lambda,\mu}:W\to\mathbb{R}$
\begin{equation} \label{e2.1}
J_{\lambda,\mu}(u,v)=\frac{1}{p} M(u,v)- F(u,v)-(\lambda+\mu)\int_{\Omega}g(x,|u|,|v|)dx,
\end{equation}
where
 \begin{equation} \label{e2.2}
\begin{gathered}
 M(u,v)=\int_{\Omega}(|\nabla u|^p+m_1(x)
|u|^p)dx+\int_{\Omega}(|\nabla v|^p+m_2(x) |v|^p)dx, \\
 F(u,v)=\int_{\partial\Omega}f(x,|u|,|v|)dx.
\end{gathered}
\end{equation}
Since
$J_{\lambda,\mu}$ is unbounded from below on whole space $W$, it is
useful to consider the functional on the Nehari manifold
\begin{equation}  \label{e2.3}
\mathcal{N}_{\lambda,\mu}(\Omega)=\{(u,v)\in W\setminus\{(0,0)\} :
\langle J_{\lambda,\mu}'(u,v),(u,v) \rangle=0\},
\end{equation}
where $\langle , \rangle$ denotes the usual duality between $W$ and
$W^{-1}$ ( $W^{-1}$ is the dual space of the Sobolev space $W$). Moreover, by definition, we
have that $(u,v)\in \mathcal{N}_{\lambda,\mu}(\Omega) $ if and only
if
\begin{equation} \label{e2.4}
\begin{split}
&M(u,v)- \int_{\partial\Omega}\big(f_u(x,|u|,|v|)|u|+f_v(x,|u|,|v|)|v|\big)dx\\
&-(\lambda+\mu)\int_{\Omega}g(x,|u|,|v|)(|u|,|v|)dx=0.
\end{split}
\end{equation}

\begin{theorem} \label{thm2.1}
$J_{\lambda,\mu}$ is coercive and bounded from below on
$\mathcal{N}_{\lambda,\mu}(\Omega)$ for $\lambda$ and $\mu$
sufficiently small.
\end{theorem}
It  can be proved that the points in
$\mathcal{N}_{\lambda,\mu}(\Omega)$ correspond to the stationary
points of the fibering map $\phi_{u,v}(t): [0,\infty)\to \mathbb{R}$
defined by $\phi_{u,v}( t )= J_{\lambda,\mu}(tu,tv)$,
 which discussed
  in Brown and Zhang \cite{4}.
Using \eqref{e2.1} for $(u,v)\in W$, we have
\begin{equation} \label{e2.5}
\begin{gathered}
\begin{aligned}
\phi_{u,v}(t)&=J_{\lambda,\mu}(tu,tv)\\
&=\frac{t^p}{p} M(u,v)-F(tu,tv)-(\lambda+\mu)\int_{\Omega}g(x,t|u|,t|v|)dx=0,
\end{aligned}\\
\begin{aligned}
\phi_{u,v}'(t) &= t^{p-1}M(u,v)- \int_{\partial\Omega}\nabla
f(x,t|u|,t|v|).(|u|,|v|)dx\\
&\quad -(\lambda+\mu)\int_{\Omega}\nabla
g(x,t|u|,t|v|).(|u|,|v|)dx,
\end{aligned}\\
\begin{aligned}
\phi_{u,v}''(t)&= (p-1)t^{p-2}M(u,v)-
\int_{\partial\Omega}\mathcal{F}(x,tu,tv)dx\\
&\quad -(\lambda+\mu)\int_{\Omega}(g_{uu}u^2+g_{vv}v^2+2g_{uv}|uv|)dx,
\end{aligned}
\end{gathered}
\end{equation}
where
\begin{equation} \label{e2.6}
\begin{gathered}
\mathcal{F}(x,tu,tv):=\frac{\partial^2}{\partial t^2}
\big(f(x,t|u|,t|v|)=f_{uu}u^2+f_{vv}v^2+2f_{uv}|uv|.
\end{gathered}
\end{equation}
We define
\begin{equation} \label{e2.7}
\begin{gathered}
\mathcal{N}_{\lambda,\mu}^{+}=\{(u,v)\in
\mathcal{N}_{\lambda,\mu}(\Omega):
\phi_{u,v}''(1)> 0\},\\
\mathcal{N}_{\lambda,\mu}^{-}=\{(u,v)\in
\mathcal{N}_{\lambda,\mu}(\Omega):
\phi_{u,v}''(1)< 0\},\\
\mathcal{N}_{\lambda,\mu}^{0}=\{(u,v)\in
\mathcal{N}_{\lambda,\mu}(\Omega): \phi_{u,v}''(1)=0\}.
\end{gathered}
\end{equation}

\begin{lemma} \label{lem2.1}
 Let $(u_{0},v_0)$ be a local minimizer for $J_{\lambda,\mu}(u,v)$
 on $\mathcal{N}_{\lambda,\mu}(\Omega)$. If
$(u_{0},v_0)$ is not in $\mathcal{N}_{\lambda,\mu}^{0}(\Omega)$,
then $(u_{0},v_0)$ is a critical point of $J_{\lambda,\mu}$.
\end{lemma}

\begin{lemma} \label{lem2.2}
There exists $K_0\subset (\mathbb{R^+})^2$ such that for all
$(\lambda ,\mu)\in K_0$, we have $\mathcal{N}_{\lambda,\mu}^{0}=
\emptyset$.
\end{lemma}
\begin{definition} \label{def2.2} \rm
A sequence $y_n=(u_{n},v_n)\subset W$ is called a Palais-Smale
sequence if $I_{\lambda,\mu}(y_{n})$ is bounded
 and $I'_{\lambda,\mu}(y_{n})\to 0$ as $n\to\infty$.
If  $I_{\lambda,\mu}(y_{n})\to c$ and $I'_{\lambda,\mu}(y_{n})\to
0$, then $y_n$ is a
 $(PS)_c$-sequence.
 It is said that the functional $I_{\lambda,\mu}$
 satisfies the Palais-Smale condition (or $(PS)_{c}$-condition), if
each Palais-Smale sequence ($(PS)_c$-sequence) has a convergent
subsequence.
\end{definition}

\begin{lemma} \label{lem2.3}
If $\{(u_n,v_n)\}$ is a $(PS)_{c}$-sequence for $J_{\lambda,\mu}$,
then $\{(u_n,v_n)\}$ is bounded in $W$ provided that
$(\lambda,\mu)\in K_1=\{(\lambda,\mu):
r-p-4r(C_8\lambda+C_9\mu){{S}}_p^p>0\}$.
\end{lemma}
\begin{lemma} \label{lem2.4}
There exists $K_2\subset \mathbb{R}^2$ such that if
 $(\lambda,\mu)\in K_2$ and $(u,v)\in N_{\lambda,\mu}^-$, then
$\int_{\partial\Omega}\mathcal{F}(x,u,v)dx>0$, where
$\mathcal{F}(x,u,v)$ is defined by \eqref{e2.6}.
 \end{lemma}

To obtain a better understanding of the behavior of fibering maps,
we will describe the nature of the derivative of the fibering maps
for all possible signs of
$\int_{\partial\Omega}\mathcal{F}(x,tu,tv)dx$ (by (A1) and
\eqref{e2.6}, $\int_{\partial\Omega}\mathcal{F}(x,tu,tv)dx$ has the
same sign for every $t>0$). Define the functions $R(t)$ and $S(t)$
as follows
\begin{gather} \label{e3.1}
R(t):=\frac{1}{p} t^pM(u,v)- F(tu,tv) \quad (t>0),\\
\label{e3.2}
S(t):=(\lambda+\mu)\int_{\Omega}g(x,t|u|,t|v|)dx
\quad (t>0),
\end{gather}
then from \eqref{e2.5} it follows  that $\phi_{u,v}(t)=R(t)-S(t)$.
Moreover, $\phi'_{u,v}(t)=0$ if and only if $R'(t)=S'(t)$, where
\begin{equation} \label{e3.3}
R'(t)=t^{p-1}M(u,v)-\int_{\partial\Omega}\Big(f_u(x,t|u|,t|v|)|u|
+f_v(x,t|u|,t|v|)|v|\Big)dx,
\end{equation}
and
\begin{equation} \label{e3.4}
S'(t)=(\lambda+\mu)\int_{\Omega}g(x,t|u|,t|v|)(|u|,|v|)dx.
\end{equation}

\begin{lemma} \label{lem3.1}
There exists $K_3\subset (\mathbb{R}^+)^2$ such that for all nonzero
$(u,v)\in W$, $\phi_{u,v}(t)$ and $\phi'_{u,v}(t)$ take on positive
values whenever $(\lambda,\mu)\in K_3$.
\end{lemma}
\begin{corollary} \label{coro3.1}
If $(\lambda,\mu)\in K_2\cap K_3$, then there exists $\varepsilon>0$
such that $J_{\lambda,\mu}(u,v)>\epsilon$ for all $(u,v)\in
\mathcal{N}_{\lambda,\mu}^{-}$.
\end{corollary}

\begin{corollary} \label{coro3.2}
for $(u,v)\in W\setminus \{(0,0)\}$ we have
\begin{itemize}
\item[(i)]
 If $\int_{\partial\Omega}\mathcal{F}(x,tu,tv)dx\leq0$, then
there exists $t_1$ such that $(t_1u,t_1v)\in N_{\lambda,\mu}^{+}$
and
 $\phi_{u,v}(t_1)<0$.
\item[(ii)]
  If $\int_{\partial\Omega}\mathcal{F}(x,tu,tv)dx\geq0$ and $(\lambda,\mu)\in K_3$, then there exist
$0<t_1< t_2$ such that $(t_1u,t_1v)\in N_{\lambda,\mu}^{+}$,
$(t_2u,t_2v)\in N_{\lambda,\mu}^{-}$ and $\phi_{u,v}(t_1)<0$.
\end{itemize}
\end{corollary}

\section{Main results}

\begin{lemma} \label{lem4.1}
\begin{itemize}
\item[(i)] For $(\lambda,\mu)\in K^*=K_0\cap K_1\cap K_3$, there exists
a minimizer of $J_{\lambda,\mu}$ on
$\mathcal{N}_{\lambda,\mu}^{+}(\Omega)$.
\item[(ii)] For $(\lambda,\mu)\in K^{**}=K_0\cap
K_1\cap K_2\cap K_3$, there exists a minimizer of
$J_{\lambda,\mu}$ on $\mathcal{N}_{\lambda,\mu}^{-}(\Omega)$.
\end{itemize}
\end{lemma}
\begin{proof}[Proof of Theorem \ref{thm1.1}]
 By Lemma \ref{lem4.1} (i) there exists
$(u_1,v_1)\in N_{\lambda,\mu}^{+}(\Omega)$ such that
$J_{\lambda,\mu}(u_1,v_1)=\inf_{(u,v) \in N_{\lambda,\mu}^{+}}
J_{\lambda,\mu}(u,v)$ and by Lemmas \ref{lem2.1} and \ref{lem2.2},
$(u_1,v_1)$ is a critical point of $J_{\lambda,\mu}$ on $W$ and
hence is a weak solution of problem \eqref{e1.1}. On the other hand
$J_{\lambda,\mu}(u,v)=J_{\lambda,\mu}(|u|,|v|)$, so we may assume
that $(u_1,v_1)$ is a positive solution and the proof is complete.
\end{proof}

\begin{proof}[Proof of Theorem \ref{thm1.2}]
 By Lemma \ref{lem4.1} there exist
 $(u_1,v_1)\in N_{\lambda,\mu}^{+}(\Omega)$ and
$(u_2,v_2)\in N_{\lambda,\mu}^{-}(\Omega)$ such that
$$
J_{\lambda,\mu}(u_1,v_1)=\inf_{(u,v) \in N_{\lambda,\mu}^{+}}
J_{\lambda,\mu}(u,v),\quad J_{\lambda,\mu}(u_2,v_2)=\inf_{(u,v) \in
N_{\lambda,\mu}^{-}}J_{\lambda,\mu}(u,v).
$$
 By Lemmas \ref{lem2.1} and \ref{lem2.2}, $(u_1,v_1)$ and $(u_2,v_2)$ are critical points of
$J_{\lambda,\mu}$ on $W$ and hence are weak solutions of problem
\eqref{e1.1}. Similar to the proof of Theorem \ref{thm1.1}, we may
assume that $(u_1,v_1)$ and $(u_2,v_2)$ are positive solutions. Also
since $N_{\lambda,\mu}^{+}\cap N_{\lambda,\mu}^{-}={\emptyset}$,
this implies that $(u_1,v_1)$ and $(u_2,v_2)$ are distinct and the
proof is complete.
\end{proof}

\end{document}